\documentclass{llncs}

\usepackage{amsmath,amssymb}

\usepackage{enumerate}

\usepackage[table,x11names]{xcolor}
\usepackage{graphicx}
\usepackage{colortbl}
\usepackage{booktabs}
\usepackage{listings}
\usepackage{pgfplots}
\usepackage{tikz}
\usetikzlibrary{shapes,arrows.meta}
\usepgfplotslibrary{external}
\usepackage{url}
\usepackage{easy-todo}

\usepackage[tableposition=top]{caption}
\usepackage{subcaption}

\usepackage[colorlinks=true]{hyperref}
\hypersetup{
  colorlinks=true,
  citecolor=SpringGreen4
}
\usepackage[nameinlink]{cleveref}

\begin{document}

\title{Bernstein-Type Bounds for Beta Distribution}


\author{Maciej Skorski}
\institute{University of Luxembourg}


\maketitle

\begin{abstract}
This work obtains a sharp closed-form exponential concentration inequality of Bernstein (sub-gamma) type for the ubiquitous beta distribution, improving upon
sub-gaussian and sub-gamma bounds previously studied in this context.

The proof leverages the novel and handy recursion of order 2 for central moments of the beta distribution, obtained from  hyper-geometric representations;
this recursion is useful for obtaining explicit expressions for central moments, as well as for developing their various approximations.
\end{abstract}


\definecolor{backcolour}{rgb}{0.95,0.95,0.92}
\definecolor{codegreen}{rgb}{0,0.6,0}
\definecolor{mymauve}{rgb}{0.58,0,0.82}
\lstdefinestyle{myStyle}{
    backgroundcolor=\color{backcolour},   
    keywordstyle=\color{blue},
    commentstyle=\color{codegreen},
    stringstyle=\color{mymauve},
    basicstyle=\footnotesize,
    breakatwhitespace=false,         
    breaklines=true,                 
    keepspaces=true,                 
    xleftmargin=10pt,
    showspaces=false,                
    showstringspaces=false,
    showtabs=false,                  
    tabsize=4,
    frame=L,
    captionpos=b, 
    abovecaptionskip=1\baselineskip,
}

\definecolor{codegreen}{rgb}{0,0.6,0}
\definecolor{codegray}{rgb}{0.5,0.5,0.5}
\definecolor{codepurple}{rgb}{0.58,0,0.82}
\definecolor{backcolour}{rgb}{0.95,0.95,0.92}
\lstdefinestyle{codestyle}{
    backgroundcolor=\color{backcolour},   
    commentstyle=\color{codegreen},
    keywordstyle=\color{magenta},
    numberstyle=\tiny\color{codegray},
    stringstyle=\color{codepurple},
    basicstyle=\footnotesize,
    breakatwhitespace=false,         
    breaklines=true,                 
    captionpos=b,                    
    keepspaces=true,                 
    numbers=left,                    
    numbersep=5pt,        
    xleftmargin=10pt,          
    showspaces=false,                
    showstringspaces=false,
    showtabs=false,                  
    tabsize=2,
}
\lstset{
	style=codestyle,
	language=Python,
	morekeywords = {as} 
}

\section{Introduction}

\subsection{Background}

The Beta distribution is ubiquitous in statistics due to its flexibility. Among many applications, it is used in analyses of uniform order statistics~\cite{jones2002fractional}, 
problems in Euclidean geometry~\cite{frankl1990some}, general theory of stochastic processes~\cite{mitov2004beta},
and applied statistical inference; the last category of applications includes time estimation in project management~\cite{clark1962pert}, hypothesis testing~\cite{zhang2002beta}, A/B testing in business~\cite{stucchio2015bayesian}, modelling in life sciences~\cite{williams1975394} and others~\cite{kim2009probabilistic,kipping2013parametrizing}.

Unfortunately, the importance and simplicity do not go in pairs. The distribution of  $X\sim\mathsf{Beta}(\alpha,\beta)$ with parameters $\alpha,\beta>0$ is given by
\begin{align}\label{eq:beta}
    \mathbf{P}\{X\leqslant \epsilon\} = 
    \int_{0}^{\epsilon} \frac{x^{\alpha-1}(1-x)^{\beta-1}}{B(\alpha,\beta)}\mbox{d} x,\ 0 \leqslant \epsilon \leqslant 1,
\end{align}
where $B(\alpha,\beta) \triangleq \int_{0}^{1}x^{\alpha-1}(1-x)^{\beta-1}\mbox{d}x$ is the normalizing constant; the integral \eqref{eq:beta}, also known as the incomplete beta function~\cite{dutka1981incomplete},
is intractable. Thus, there is a strong demand for \emph{closed-form approximations} of tail probabilities, for example in the context of adaptive Bayesian inference~\cite{elder2016bayesian}, Bayesian non-parametric statistics~\cite{castillo2017polya}, properties of random matrices~\cite{frankl1990some,perry2020statistical}, and (obviously) large deviation theory~\cite{zhang2020non}.

The main goal of this paper is to give accurate closed-form bounds for the beta distribution in the form of \emph{sub-gamma type exponential concentration inequalities}~\cite{boucheron2003concentration}, that is,
\begin{align}\label{eq:bernstein}
 \mathbf{P}\{X-\mathbf{E}[X]<-\epsilon\},\mathbf{P}\{X-\mathbf{E}[X]>\epsilon\} \leqslant \exp\left(-\frac{\epsilon^2}{2v+2c\epsilon}\right)
\end{align}
for any $\epsilon\geqslant 0$ and the two parameters:
the \emph{variance proxy}
 $v$ and the \emph{scale} $c$, both depending on $\alpha,\beta$. Such concentration bounds, pioneered by Bernstein~\cite{bernstein1946theory} and popularized by Hoeffding~\cite{hoeffding1994probability}, are capable of modelling both the sub-gaussian and sub-exponential behaviours. Due to this flexibility, bounds of this sort are the working horse of  approximation arguments used in modern statistics~\cite{kahane1960proprietes,boucheron2003concentration,maurer2021concentration,wainwright2019high}.

\subsection{Contribution}

The contributions of this work are as follows:
\begin{itemize}
    \item \textbf{Optimal Bernstein-type concentration inequality}. We establish an exponential concenration bound \eqref{eq:bernstein} with $v$ that matches the variance (which is optimal) and, respectively, the optimal value of $c$ which depends on the ratio of the third and second moment. This bound is shown optimal in the regime of small deviations.
    \item \textbf{Useful recursion for central moments} of the Beta distribution: it is of order 2 with coefficients linearly dependent on the moment order. This formula addresses the lack of a simple closed-form formula for higher-order moments. We use it to estimate the moment generating function. 
    \item \textbf{Implementation and numerical evaluation}. 
    Symbolic algebra software has been used to validate numerical comparison, and a numerical evaluation is provided to illustrate the new bounds. The code snippets are shared in the paper, and the evaluation is also shared as a Colab notebook~\cite{mskorski_beta}.
\end{itemize}

\subsection{Related Work}

When judging the bounds in the form of \eqref{eq:bernstein}, it is important to insist on
the optimality of the sub-gaussian behavior. For small deviations $\epsilon$ the bound \eqref{eq:bernstein} becomes approximately gaussian with variance $v$, thus we ideally want $v^2=\mathbf{Var}[X]$ and exponential bounds of this type are considered optimal in the literature~\cite{ben2017concentration}. On the other hand, bounds with  $v^2>\mathbf{Var}[X]$ essentially overshoot the variance, leading to unnecessary wide tails and increasing uncertainty in statistical inference. 

Bearing this in mind, we review prior suboptimal bounds in the form of \eqref{eq:bernstein}:
\begin{itemize}
    \item Folklore methods give some crude bounds, for example one can express a beta random variable in terms of gamma distributions and utilize their concentration properties; such techniques do not give optimal exponents.
    \item Some bounds in the form of \eqref{eq:bernstein} can be derived from established inequalities on the incomplete beta function. Using the well-known inequality \cite{dumbgen1998new}, which depends on a Kullback-Leibler term, unfortunately leads to suboptimal bounds in the regime of small deviations (as seen from the Taylor approximation, we necessarily obtain the variance factor of $v^2 =\frac{\alpha+\beta+1}{\alpha+\beta}\cdot\mathbf{Var}[X]$, hence overshooting the variance). 
    \item The work \cite{frankl1990some} gives bounds with explicit but suboptimal $v$ and $c$, only valid in a limited range of deviations $\sqrt{\frac{6\alpha}{(\alpha+\beta)^2}} < \epsilon < \frac{\alpha}{\alpha+\beta}$.
    The proof relies on specific integral estimates, and cannot be sharpened much.
    \item The work \cite{marchal2017sub} determines the best bound assuming $c=0$ (that is, of sub-gaussian type). They are not in a closed form, but can be numerically computed as a solution of a transcendental equation. 
    Although the bound is quite sharp for the symmetric case $\alpha=\beta$, it is much worse than our bound when the beta distribution is skewed (the typical case). 
    \item The work \cite{zhang2020non} obtains suboptimal $v$ and $c$, 
    shown to be far from the true values by unknown constant factors. With these techniques, it is not possible to obtain the optimal exponent, which is the focus of this work.
\end{itemize}
As for the central moment recursion, we note that:
\begin{itemize}
    \item Little is known about formulas for higher central moments (as opposed to raw moments, simpler to compute but not useful for concentration inequalities). Textbooks do not discuss neither explicit formulas nor derivation algorithms beyond the order of 4 (skewness and kurtosis)
    \item The modern literature~\cite{johnson1995continuous,gupta2004handbook} credits the recursive formulas found in ~\cite{muhlbach1972rekursionsformeln}. Unfortunately, that recursion is computationally inefficient due to its unbounded depth and too complicated to be manipulated for the task of closed-form moment estimation.
\end{itemize}

\begin{remark}
While the current paper was under review, Henzi and Duembgen \cite{henzi2022some} presented similar tail bounds, obtained by working with integral-derived inequalities. When framed as Bernstein's inequality, their constant $c$ is worse for the heavier tail (e.g. the right tail in case $\alpha<\beta$), but better on the lighter tail (e.g. $c<0$ for the left tail in case $\alpha<\beta$).
\end{remark}

\subsection{Organization}

The remainder of the paper is organized as follows:
\Cref{sec:results} presents 
the results, \Cref{sec:prelim} provides the technical background, \Cref{sec:proofs} gives proofs and \Cref{sec:conclude} concludes the work.
The implementation and numerical evaluation are shared in the notebook~\cite{mskorski_beta}.

\section{Results}\label{sec:results}

\subsection{Optimal Bernstein-Type Concentration Bounds}

The central result of this work is a Bernstein-type tail bound with numerically optimal parameters. The precise statement is given in the following theorem.
\begin{theorem}[Bernstein's Inequality]\label{thm:beta_bernstein}
Let $X\sim\mathsf{Beta}(\alpha,\beta)$. Define the parameters 
\begin{align}
\begin{aligned}
v & \triangleq {\frac{\alpha\beta}{(\alpha+\beta)^2(\alpha+\beta+1)}}  \\
c & \triangleq  \frac{2 \left(\beta - \alpha\right)}{\left(\alpha + \beta\right) \left(\alpha + \beta + 2\right)}. \\
\end{aligned}
\end{align}
Then the upper tail of $X$ is bounded as 
\begin{align}
\mathbf{P}\left\{X > \mathbf{E}[X]+\epsilon \right\}  \leqslant 
\begin{cases}
\exp\left(-\frac{\epsilon^{2}}{2 \left(v+\frac{c \epsilon}{3} \right)}\right) & \beta\geqslant \alpha \\
\exp\left(-\frac{\epsilon^{2}}{2v} \right) & \beta < \alpha,
\end{cases}
\end{align}
and  the lower tail of $X$ is bounded as
\begin{align}
\mathbf{P}\left\{X < \mathbf{E}[X]-\epsilon \right\}  \leqslant 
\begin{cases}
\exp\left(-\frac{\epsilon^{2}}{2 \left(v+\frac{c \epsilon}{3} \right)}\right) & \alpha \geqslant \beta \\
\exp\left(-\frac{\epsilon^{2}}{2v} \right) & \alpha < \beta.
\end{cases}
\end{align}
\end{theorem}

\begin{remark}[Variance and Scale Parameters]
The variance parameter equals $v=\mathbf{Var}[X]$, and the scale parameter equals $c=\frac{\mathbf{E}[(X-\mathbf{E}[X])^3]}{\mathbf{Var}[X]}$.
\end{remark}

\begin{remark}[Mixed Gaussian-Exponential Behaviour]
For small values of $\epsilon$ the bound behaves like gaussian with variance $v=\mathbf{Var}[X]$. For bigger values of $\epsilon$, the bound is close to the exponential tail with parameter
$\frac{2c}{3v}$.
\end{remark}

The result below shows that the parameters $c$ and $v$ are best possible, in the regime of small deviations $\epsilon$, for the exponential moment method:

\begin{theorem}[Best Cram\'{e}r-Chernoff Bound]\label{thm:lower_bound}
The best bound that can be obtained from the Cramer-Chernoff method is
\begin{align}
\mathbf{P}\left\{X > \mathbf{E}[X]+\epsilon \right\}  \leqslant \mathrm{e}^{-\frac{\epsilon^{2}}{2 v} + \frac{c \epsilon^{3}}{6 v^{2}} + O\left(\epsilon^{4}\right)},
\end{align}
as $\epsilon \to 0$, with constants $c,v$ as in \Cref{thm:beta_bernstein}.
\end{theorem}
Comparing this result with the tail from \Cref{thm:beta_bernstein} as $\epsilon\to0$ we obtain: 
\begin{corollary}[Optimality of \Cref{thm:beta_bernstein}]
The values of constants $c$ and $v$ in the Bernstein-type inequality in \Cref{thm:beta_bernstein} are optimal. 
\end{corollary}

\subsection{Handy Recurrence for Central Moments}

Our optimal Bernstein inequality is proved using the novel recursion for central moments, presented below in \Cref{lemma:beta_moments_recurrence} as the contribution of independent interest. Being of order 2 it is not only efficient to evaluate numerically, but also easy to manipulate algebraically when working with closed-form formulas.
\begin{theorem}[Order 2 Recurrence for Central Moments]\label{lemma:beta_moments_recurrence}
For $X\sim\mathsf{Beta}(\alpha,\beta)$ and any integer order $d\geqslant 2$, the following recurrence relation holds:
\begin{align}
\begin{split}
  \mathbf{E}[(X-\mathbf{E}[X])^d] &= \frac{(d-1)(\beta-\alpha)}{(\alpha+\beta)(\alpha+\beta+d-1)}\cdot\mathbf{E}[(X-\mathbf{E}[X])^{d-1}] \\
  &\quad + \frac{(d-1)\alpha\beta}{(\alpha+\beta)^2(\alpha+\beta+d-1)}\cdot \mathbf{E}[(X-\mathbf{E}[X])^{d-2}].
 \end{split}
\end{align}
\end{theorem}
\begin{corollary}[P-recursive Property]
The recursion given in \Cref{lemma:beta_moments_recurrence}, after rearrangements, is of order $2$ with coefficients linear in the index $d$.
\end{corollary}

The implementation of the recursion, performed in Python's symbolic algebra package \texttt{Sympy} \cite{sympy},
is presented in \Cref{alg:recursion}. Scalability is ensured by \emph{memoization}, which caches intermediate results to avoid repeated calls.

\begin{lstlisting}[
    caption={Efficient algorithm finding exact formulas for central moments.}, 
    label=alg:recursion,
    language=Python,
    float,
]
from functools import cache

@cache
def beta_central_moment(d,a,b):
  """ find the central moment of order d for Beta(a,b) """
  if d == 0:
    return 1
  elif d == 1:
    return 0
  else:
    c1 = (d-1)*(b-a)/((a+b)*(a+b+d-1))
    c2 = (d-1)*a*b/((a+b)**2*(a+b+d-1))
    return c1*beta_central_moment(d-1,a,b)+c2*beta_central_moment(d-2,a,b)

# usage: find the variance
import sympy as sm
a,b = sm.symbols('alpha beta')
beta_central_moment(2,a,b)
\end{lstlisting}
To demonstrate the algorithm in action, we list some first central moments of the beta distribution in \Cref{tab:moments}. Note that the algorithm addresses the lack of simple derivations for such formulas, as well as the lack of formulas beyond skewness and kurtosis.

\begin{table}[!ht]
\begin{tabular}{ll}
\toprule
Central Moment & Explicit Formula \\
\midrule
  $\mathbf{E}[X]$ & $\frac{\alpha}{\alpha+\beta}$ \\
  $\mathbf{Var}[X]$ & $\frac{\alpha\beta}{(\alpha+\beta)^2(\alpha+\beta+1)}$\\
  $\textbf{Skew}[X]$   & $\frac{2 \left(\beta- \alpha \right) \sqrt{\alpha + \beta + 1}}{\sqrt{\alpha \beta} \left(\alpha + \beta + 2\right)}$  \\
  $\textbf{Kurt}[X]$   &  
  $\frac{3 \left(\alpha \beta \left(\alpha + \beta + 2\right) + 2 \left(\alpha - \beta\right)^{2}\right) \left(\alpha + \beta + 1\right)}{\alpha \beta \left(\alpha + \beta + 2\right) \left(\alpha + \beta + 3\right)}$ \\
  $\mathbf{E}[ (X-\mathbf{E}[X])^5 ] / \sqrt{\mathbf{Var}[X]}^5 $ & $ \frac{4 \left(\beta - \alpha\right) \left(\alpha + \beta + 1\right)^{\frac{3}{2}} \left(3 \alpha \beta \left(\alpha + \beta + 2\right) + 2 \alpha \beta \left(\alpha + \beta + 3\right) + 6 \left(\alpha - \beta\right)^{2}\right)}{\alpha^{\frac{3}{2}} \beta^{\frac{3}{2}} \left(\alpha + \beta + 2\right) \left(\alpha + \beta + 3\right) \left(\alpha + \beta + 4\right)}$ \\
\bottomrule
\end{tabular}
\caption{Central beta moments, generated from \Cref{lemma:beta_moments_recurrence} using \Cref{alg:recursion}.}
\label{tab:moments}
\end{table}

Finally, we note that the recurrence implies, by induction, the following important property:
\begin{corollary}[Skewness of Beta Distribution]\label{lemma:skewness}
The odd central moments are of the same sign as the value of $\beta-\alpha$.
\end{corollary}

\subsection{Numerical Evaluation}\label{sec:results:evaluation}

The experiment summarized in \Cref{fig:eval} (the code shared in \cite{mskorski_beta}) demonstrates the advantage of sub-gamma bounds obtained in \Cref{thm:beta_bernstein} over the optimal sub-gaussian bounds from prior work~\cite{marchal2017sub}.
\begin{figure}[h!]
    \centering
\begin{subfigure}[t]{0.45\textwidth}
\resizebox{1.0\linewidth}{!}{
\begin{tikzpicture}
\begin{axis}[xlabel=$\epsilon$,ylabel=$\mathbf{P}\left\{X-\mathbf{E}X>\epsilon\right\}$,
axis x line*=bottom,
axis y line*=left,
scaled x ticks = false,
xticklabel style={
        /pgf/number format/fixed,
        /pgf/number format/precision=2
},
legend pos= north east
]
\addplot[line width=1,dashed,black] table [y=subgamma,x=eps,col sep=comma]{beta_2_98.csv};
\addplot[line width=1,red,solid,label=gaussian] table [y=subgauss,x=eps,col sep=comma]{beta_8_92.csv};
\addlegendentry{sub-gaussian}
\addlegendentry{sub-gamma \textbf{(this work)}}
\end{axis}
\end{tikzpicture}
}
\caption{$X\sim\mathsf{Beta}(2,98)$}
\end{subfigure}
\begin{subfigure}[t]{0.45\textwidth}
\resizebox{1.0\linewidth}{!}{
\begin{tikzpicture}
\begin{axis}[xlabel=$\epsilon$,ylabel=$\mathbf{P}\left\{X-\mathbf{E}X>\epsilon\right\}$,
axis x line*=bottom,
axis y line*=left,
scaled x ticks = false,
xticklabel style={
        /pgf/number format/fixed,
        /pgf/number format/precision=2
},
legend pos= north east
]
\addplot[line width=1,dashed,black] table [y=subgamma,x=eps,col sep=comma]{beta_2_98.csv};
\addplot[line width=1,solid,red] table [y=subgauss,x=eps,col sep=comma]{beta_8_92.csv};
\addlegendentry{sub-gaussian}
\addlegendentry{sub-gamma (\textbf{this work})}
\end{axis}
\end{tikzpicture}
}
\caption{$X\sim\mathsf{Beta}(8,92)$}
\end{subfigure}
\caption{Numerical evaluation of the best sub-gamma bounds (\Cref{thm:beta_bernstein}) and the best sub-gaussian bounds (\cite{marchal2017sub}).}
\label{fig:eval}
\end{figure}

For skewed beta distributions, the bound from this work is more accurate than the sub-gaussian approximation. The chosen range of parameters covers the cases where the expectation $\mathbf{E}[X]$ is a relatively small number like $0.02$ or $0.2$, a typical range for many practical Beta models, particularly A/B testing. Note that there is still room for improvement in the regime of larger deviations $\epsilon$, where the bounds could potentially benefit from refining the numeric value of $c$. The experiment details are shared in the Colab Notebook~\cite{mskorski_beta}.

\section{Preliminaries}\label{sec:prelim}

\subsection{Gaussian Hypergeometric Function}
The gaussian hypergeometric function is defined as follows~\cite{gauss1813disquisitiones}:
\begin{align}\label{eq:hypergeom}
    _2 F_1(a,b;c;,z) = \sum_{k=0}^{+\infty}\frac{(a)_k(b)_k}{(c)_k}\cdot \frac{z^k}{k!},
\end{align}
where we use the Pochhammer symbol defined as
\begin{align}
(x)_k = 
\begin{cases}
1 & k = 0 \\
x(x+1)\cdots (x+k-1) &  k>0 .
\end{cases}
\end{align}
We call two functions $F$ of this form \emph{contiguous} when their parameters differ by integers.
Gauss considered $_2 F_1(a',b';c';,z)$ where $a'=a\pm 1,b'=b\pm 1,c'=c\pm 1$ and proved that between $F$ and any two of these functions, there exists a linear relationship with coefficients linear in $z$. 
It follows~\cite{vidunas2003contiguous,ibrahim2012contiguous} that $F$ and any two of its contiguous series are linearly dependent, with the coefficients being rational in parameters and $z$. For our purpose, we need 
to express $F$ by the series with increased second argument. The explicit formula comes from~\cite{mathworld_hypergeom}:

\begin{lemma}[Hypergeometric Contiguous Recurrence]\label{lemma:hypergeom_recur}
The following recurrence holds for the gaussian hypergeometric function:
\begin{align}\label{eq:hypergeom_recur}
\begin{split}
 _2 F_1(a,b;c;,z)& =\frac{2b-c+2+(a-b-1)z}{b-c+1}\, _2F_1(a,b+1;c;,z) \\
 &\quad + \frac{(b+1)(z-1)}{b-c+1}\, _2F_1(a,b+2;c;,z).
\end{split}
\end{align}
\end{lemma}

\subsection{Beta Distribution Properties}
We use the machinery of hypergeometric functions to establish certain properties of the beta distribution. The first result expresses the central moment in terms of the gaussian hypergeometric function.
\begin{lemma}[Central Beta Moments]\label{lemma:beta_moments}
Let $X\sim\mathsf{Beta}(\alpha,\beta)$, then we have that: 
\begin{align}
   \mathbf{E}[(X-\mathbf{E}[X])^d] = \left(-\frac{\alpha}{\alpha+\beta}\right)^d\ 
    _2F_1\left(\alpha,-d;\alpha+\beta;\frac{\alpha+\beta}{\alpha}\right),
\end{align}
where $_2F_1$ is the gaussian hypergeometric function.
\end{lemma}
The proof appears in \Cref{proof:lemma:beta_moments}.

\subsection{Cram\'{e}r-Chernoff method}

Below, we review the canonical method of obtaining concentration inequalities from the moment generating function, following the discussion in \cite{boucheron2003concentration}.

Recall that for a given random variable $Z$ we define its moment generating function (MGF) and, respectively, cumulant generating function as
\begin{align}
\begin{aligned}
\phi_Z(t) & \triangleq \mathbf{E}\left[ \exp(tZ) \right] \\
\psi_Z(t) & \triangleq \log\mathbf{E}\left[ \exp(tZ) \right].  
\end{aligned}
\end{align}
In this notation, Markov's exponential inequality becomes
\begin{align}
\mathbf{P}\left\{ Z \geqslant \epsilon \right\} \leqslant 
\exp(-t\epsilon)\mathbf{E}\left[ \exp(tZ) \right] = \exp(-t\epsilon + \psi_Z(t)),
\end{align}
valid for any $\epsilon$ and any non-negative $t$, and non-trivial when $\phi_Z(t)$ is finite. The optimal value of $t$ is determined by the so called Cram\'{e}r transform:
\begin{align}
\psi^{*}(\epsilon) = \sup\left\{t\epsilon -\psi_Z(t): t\geqslant 0 \right\},
\end{align}
and leads to the Chernoff inequality
\begin{align}\label{eq:chernoff_bound}
\mathbf{P}\left\{ Z \geqslant \epsilon \right\}  \leqslant \exp(-\psi^{*}_Z(\epsilon) ).
\end{align}


\subsection{Logarithmic Inequalities}

It is well known that $\log(1+x) = x+\frac{x^2}{2}+\ldots$ for non-negative $x$. Below, we present a useful refinement of this expansion:
\begin{lemma}[Padé approximation of logarithm]\label{lemma:log_ineq}
For any $x\geqslant 0$ we have that
\begin{align}\label{eq:log_inequality}
\log(1+x) \leqslant x-\frac{x^2}{2\left(1+\frac{2 x}{3}\right)}.
\end{align}
\end{lemma}
The bound is illustrated in \Cref{fig:log_ineq} below, and we see that it matches up to the term $O(x^4)$. The proof appears in \Cref{proof:lemma:log_ineq}.

\begin{figure}[ht!]
\resizebox{\linewidth}{!}{
\begin{tikzpicture}
\begin{axis}[
domain=0:3,
axis lines=left,
no marks,
xtick = {0,1,2,3},
ytick = {1},
xlabel = {$x$},
legend style={draw=none, anchor= north west}
]
\addplot[black] {ln(1+x)};
\addlegendentry{$\log(1+x)$};
\addplot[black, dashed] {x-x^2/(2*(1+2*x/3)) };
\addlegendentry{$x-\frac{x^2}{2\left(1+\frac{2x}{3}\right)}$};
\end{axis}
\end{tikzpicture}
}
\caption{The logarithmic inequality from \Cref{lemma:log_ineq}.}
\label{fig:log_ineq}
\end{figure}



\section{Proofs}\label{sec:proofs}

\subsection{Proof of \Cref{lemma:beta_moments}}\label{proof:lemma:beta_moments}
We know that the raw higher-order moments of $X$ are given by~\cite{johnson1995continuous}
\begin{align}
    \mathbf{E}[X^d] = \frac{(\alpha)_d}{(\alpha+\beta)_d}.
\end{align}
Combining this with the binomial theorem, we obtain
\begin{align}
\begin{split}
        \mathbf{E}[(X-\mathbf{E}[X])^d] & = \sum_{k=0}^{d}(-1)^{d-k}\binom{d}{k}\mathbf{E}[X^k]\left(\frac{\alpha}{\alpha+\beta}\right)^{d-k} \\
        & = \sum_{k=0}^{d}(-1)^{d-k}\binom{d}{k}\frac{(\alpha)_k}{(\alpha+\beta)_k}\left(\frac{\alpha}{\alpha+\beta}\right)^{d-k}.
\end{split}
\end{align}
Finally, by $\binom{d}{k} = (-1)^k\frac{(-d)_k}{k!}$ and the definition of the hypergeometric function:
\begin{align}
\begin{split}
        \mathbf{E}[(X-\mathbf{E}[X])^d] & =
         (-1)^d\sum_{k=0}^{d}\frac{(-d)_k (\alpha)_k}{k!(\alpha+\beta)_k}\left(\frac{\alpha}{\alpha+\beta}\right)^{d-k} \\
        & = \left(-\frac{\alpha}{\alpha+\beta}\right)^d\sum_{k=0}^{d}\frac{(-d)_k (\alpha)_k}{k!(\alpha+\beta)_k}\left(\frac{\alpha+\beta}{\alpha}\right)^{k} \\
        & = \left(-\frac{\alpha}{\alpha+\beta}\right)^d\ _2F_1\left(\alpha,-d;\alpha+\beta;\frac{\alpha+\beta}{\alpha}\right),
\end{split}
\end{align}
which finishes the proof.

\subsection{Proof of \Cref{lemma:beta_moments_recurrence}}
The goal is to prove the recursion formula \eqref{eq:hypergeom_recur}. One way to accomplish this is to reuse the summation formulas developed in the proof of \Cref{lemma:beta_moments}. Below, we give an argument that uses the properties of hypergeometric functions to better highlight their connection to the beta distribution.

To simplify the notation, we define $a=\alpha$, $b=-d$, $c=\alpha+\beta$, and $z=\frac{\alpha+\beta}{\alpha}$.
Define $\mu_d \triangleq \mathbf{E}[(X-\mathbf{E}[X])^d]$. Then
using \Cref{lemma:beta_moments} and \Cref{lemma:hypergeom_recur} we obtain:
\begin{align}
\begin{split}
(-z)^d\cdot \mu_d  & = _2F_1(a,b;c;z) \\
    & = \frac{2b-c+2+(a-b-1)z}{b-c+1}\, _2F_1(a,b+1;c;,z) \\
 &\quad + \frac{(b+1)(z-1)}{b-c+1}\, _2F_1(a,b+2;c;,z).
\end{split}
\end{align}
In terms of $\alpha,\beta,d$ we obtain:
\begin{align}
\begin{aligned}
\frac{2b-c+2+(a-b-1)z}{b-c+1} &= (d-1)\cdot \frac{\alpha-\beta}{\alpha(\alpha+\beta+d-1)}\\
\frac{(b+1)(z-1)}{b-c+1} & = (d-1)\cdot \frac{\beta}{\alpha(\alpha+\beta+d-1)}.
\end{aligned}
\end{align}
The computations are done in \texttt{SymPy} package~\cite{sympy}, as shown in \Cref{alg:hypergeom_coeffs}.
\begin{lstlisting}[
language=Python,caption={Simplifying Hypergeometric Recurence},label={alg:hypergeom_coeffs}
]
from sympy.abc import a,b,c,z,d,alpha,beta

p = ( 2*b-c+2 + (a-b-1)*z )/(b-c+1)
q = (b+1)*(z-1) / (b-c+1)

subs = {a:alpha,b:-d,c:alpha+beta,z:(alpha+beta)/alpha}

print( p.subs(subs).factor() )
print( q.subs(subs).factor() )
\end{lstlisting}

Since we have
\begin{align}
\begin{aligned}
    _2F_1(a,b+1;c;,z)& =\ _2F_1(a,-d+1;c;,z) = \mu_{d-1} \cdot (-z)^{d-1}\\
    _2F_1(a,b+2;c;,z)& =\ _2F_1(a,-d+2;c;,z) = \mu_{d-2} \cdot (-z)^{d-2},
\end{aligned}
\end{align}
it follows that
\begin{align}
    z^2\mu_d  = -\frac{z(d-1)(\alpha-\beta)}{\alpha(\alpha+\beta+d-1)}\cdot \mu_{d-1} + \frac{(d-1)\beta}{\alpha(\alpha+\beta+d-1)}\cdot \mu_{d-2} .
\end{align}
Recalling that $z=\frac{\alpha+\beta}{\alpha}$ we finally obtain
\begin{align}
    \mu_d = -\frac{(d-1)(\alpha-\beta)}{(\alpha+\beta)(\alpha+\beta+d-1)}\cdot \mu_{d-1} + \frac{(d-1)\alpha\beta}{(\alpha+\beta)^2(\alpha+\beta+d-1)}\cdot \mu_{d-2},
\end{align}
which finishes the proof.

\subsection{Proof of  \Cref{lemma:log_ineq}}\label{proof:lemma:log_ineq}
Consider the function $f(x) = \log(1+x) - \left( x-\frac{x^2}{2\left(1+\frac{2x}{3}\right)}\right)$ for non-negative $x$.
Using \texttt{Sympy} we find that
\begin{align}
f'(x) = - \frac{x^{3}}{\left(x + 1\right) \left(2 x + 3\right)^{2}},
\end{align}
as demonstrated in \Cref{alg:log_ineq}. Thus, $f$ is non-decreasing; since $f(0)=0$ we obtain $f(x)\leqslant 0$ as claimed.

\begin{lstlisting}[
    caption={A Padé-type logarithm inequality.}, 
    label=alg:log_ineq,
]
import sympy as sm
x = sm.symbols('x')
fn = sm.log(1+x)-x+x**2/(2*(1+2*x/3))
sm.print_latex(fn.diff(x,1).factor())
\end{lstlisting}

\subsection{Proof of \Cref{thm:beta_bernstein}}\label{sec:proof:beta_bernstein}

If $X\sim\mathsf{Beta}(\alpha,\beta)$, then for $X'=1-X$ we obtain $X-\mathbf{E}[X] =  \mathbf{E}[X']-X' $ and thus the upper tail of $X$ equals the lower tail of $X'$, and the lower tail of $X$ equals the upper tail of $X'$. Furthermore, from \eqref{eq:beta} it follows that $X'\sim \mathsf{Beta}(\beta,\alpha)$. Thus, by exchanging the roles of $\alpha$ and $\beta$ accordingly, it suffices to prove the theorem for the upper tail. 

To facilitate calculations, we introduce the centred version of $X$:
\begin{align}
Z \triangleq X-\mathbf{E}[X].
\end{align}
With the help of normalized moments
\begin{align}
m_d \triangleq \frac{\mathbf{E}[Z^d]}{d!} ,
\end{align}
we expand the moment generating function of $Z$ into the following series 
\begin{align}\label{eq:mfg_series}
\phi(t) \triangleq \mathbf{E}\exp(tZ) =  1 + \sum_{d=2}^{+\infty} m_d t^d,\quad t\in\mathbb{R},
\end{align}
(which converges everywhere as $|Z|\leqslant 1$), expand its derivative as 
\begin{align}\label{eq:mfg_diff_series}
\phi'(t) =   \sum_{d=2}^{+\infty} d m_{d} t^{d-1},\quad t\in\mathbb{R},
\end{align}
(everywhere, accordingly). 

The proof strategy is as follows: we start by expressing the cumulant generating function as an integral involving the moment generating function and its derivative;
then we study series expansions of the functions involved, using the moment recursion; this leads to a tractable upper-bound on the integral; finally, the cumulant bound is optimized as 
Cram\'{e}r-Chernoff  method. To achieve the promised optimal constants, it is critical to keep all intermediate estimates sharp up to the order of 4. The proof flow is illustrated on \Cref{fig:roadmap}.

\begin{figure}[ht!]
\begin{tikzpicture}[
every node/.style = {draw=black, rounded corners, fill=gray!30, 
         minimum width=9cm,
         align=center
},
]

\node [draw, rectangle] at (0,0) (cumulant_integral) {Cumulant generating function as integral in $\phi,\phi'$ (\Cref{claim:cumulant_integral}) };
\node [draw, rectangle] at (0,-1) (cumulant_series) {Recursion for series coefficients of $\phi',\phi$ (\Cref{claim:cumulant_series_recursion})};
\node [draw, rectangle] at (0,-2) (diff_cumulant_bound) {Rational upper bound on ${\phi'}/{\phi}$ from recursion (\Cref{claim:diff_cumulant_bound})};
\node [draw, rectangle] at (0,-3) (cumulant_bound) {Tractable bound on  cumulant generating function (\Cref{claim:cumulant_bound})};
\node [draw, rectangle] at (0,-4) (chernoff_bound) {Cram\'{e}r-Chernoff optimization on the bound (\Cref{claim:cumulant_optimisation})};

\draw[-latex] (cumulant_integral) -- (cumulant_series);
\draw[-latex] (cumulant_series) -- (diff_cumulant_bound);
\draw[-latex] (diff_cumulant_bound) -- (cumulant_bound);
\draw[-latex] (cumulant_bound) -- (chernoff_bound);
\end{tikzpicture}
\caption{The proof roadmap.}
\label{fig:roadmap}
\end{figure}

Observe that the cumulant generating function $\psi$ can be expressed in terms of the moment generating function $\phi$ as:
\begin{claim}[Cumulant generating function as integral]\label{claim:cumulant_integral}
We have
\begin{align}
\psi(t)  \triangleq \log \phi(t) = \int_0^{t} \frac{\phi'(s)}{\phi(s)}\mbox{d} s.
\end{align}
\end{claim}
\begin{proof}
This follows by the logarithmic derivative identity: $\log\phi' = \frac{\phi'}{\phi}$.
\end{proof}

We now present a recursion for coefficients of the series expansions: 
\begin{claim}\label{claim:cumulant_series_recursion}
The coefficients of the moment generating series satisfy
\begin{align}\label{eq:recurssion_scaled}
d(\alpha+\beta+d-1)m_d = \frac{(d-1)(\beta-\alpha)}{\alpha+\beta}m_{d-1}+\frac{\alpha\beta}{(\alpha+\beta)^2}m_{d-2},\quad d\geqslant 2.
\end{align}
\end{claim}
\begin{proof}[Proof of Claim]
This follows from \Cref{lemma:beta_moments_recurrence} by simple manipulations.
\end{proof}

\begin{remark}
Denote $p=\frac{\alpha}{\alpha+\beta}$ and $n=\alpha+\beta$ then we can write
\begin{align}
    m_d = \frac{(d-1)(1-2p)}{d(d+n-1)} m_{d-1} + \frac{p(1-p)}{d(d+n-1)} m_{d-2}
\end{align}
\end{remark}

The relation between $\phi'$ and $\phi$ is given in the following auxiliary result.

\begin{claim}\label{claim:diff_cumulant_bound}
For non-negative $t$ and $c=\frac{\beta-\alpha}{(\alpha+\beta)(\alpha+\beta+2)}$ it holds that
\begin{align}
\begin{aligned}
\frac{\phi'(t)}{\phi(t)}  \leqslant 
\frac{vt}{1-c t}, && \alpha \leqslant \beta\text{ and } t<\frac{1}{c}\\
\phi(t) \leqslant \exp\left(\frac{vt^2}{2}\right), && \alpha>\beta.
\end{aligned}
\end{align}
\end{claim}

\begin{proof}[Proof of Claim]

The proof splits depending on the relation between $\alpha$ and $\beta$. The case $\alpha>\beta$ is a bit easier and leads to sub-gaussian bounds.

\underline{Case $\alpha > \beta$}: 
From \Cref{lemma:skewness} and \eqref{eq:recurssion_scaled} it follows that $m_d$ is non-negative when $d$ is even, and negative otherwise. Thus, for even $d$ we have:
\begin{align}
m_d \leqslant \frac{1}{d}\cdot \frac{\alpha\beta}{\alpha+\beta+d-1} m_{d-2}\leqslant \frac{v}{d}\cdot m_{d-2}.
\end{align}
Repeating this $d/2$ times and combining with $m_{d}\leqslant 0$ for odd $d$, we obtain
\begin{align}
    m_d \leqslant
    \begin{cases}
     \frac{v^{\frac{d}{2}}}{d!!} & d \text{ even} \\
     0 & d \text{ odd}.
    \end{cases}
\end{align}
Using $d!! = 2^{d/2}(d/2)!$ for even $d$, for $t\geqslant 0$ we obtain
\begin{align}
\phi(t) &\leqslant 1+\sum_{d=2}^{+\infty}m_d t^d \leqslant \exp\left(\frac{vt^2}{2}\right),
\end{align}
as required.

\underline{Case $\alpha\leqslant \beta$}: 
By \eqref{eq:mfg_series} and \eqref{eq:mfg_diff_series} for any $t$ and any $c$ (to be determined later)
\begin{align}\label{eq:diff_opeator}
(1-ct){\phi'(t)} - vt \phi(t) = \sum_{d=3}^{+\infty} \left( d m_d - (d-1) c m_{d-1}- vm_{d-2} \right) t^{d-1},
\end{align}
where we use the fact that the expansion terms with $1,t,t^2$ vanish. 

Using \eqref{eq:recurssion_scaled} to eliminate the term with $m_{d-2}$, and expressing $v$ in terms of $\alpha,\beta$,  we obtain
\begin{multline}\label{eq:operator_coeffs}
d m_d - (d-1) c m_{d-1}- vm_{d-2}  \\ =
 - \frac{d \left(d - 2\right)}{\alpha + \beta + 1} m_d 
 +\frac{\left(1 - d\right) \left(\alpha - \beta + c \left(\alpha^{2} + 2 \alpha \beta + \alpha + \beta^{2} + \beta\right)\right)}{\alpha^{2} + 2 \alpha \beta + \alpha + \beta^{2} + \beta}  m_{d-1}.
\end{multline}

Since we assume $t\geqslant 0$, it suffices to show that \eqref{eq:operator_coeffs} is non-positive for $d\geqslant 3$, for $c$ defined as in \Cref{thm:beta_bernstein}.

Since $\alpha\leqslant \beta$, we have that $m_d\geqslant 0$ for all $d$ by induction in \eqref{eq:recurssion_scaled}. In particular, $m_{d-2}\geqslant 0$, and from \eqref{eq:recurssion_scaled} we obtain the bound
\begin{align}
m_d\geqslant \frac{\left(\beta-\alpha \right) \left(d - 1\right) {m}_{d - 1}}{d \left(\alpha + \beta\right) \left(\alpha + \beta + d - 1\right)},
\end{align}
(note that the bound is sharp for $d=3$) which together with \eqref{eq:operator_coeffs} yields
\begin{multline}
d m_d - (d-1) c m_{d-1}- vm_{d-2}  \leqslant \\
c \left(1 - d\right) {m}_{d - 1} + \frac{\left(- \alpha d + \alpha + \beta d - \beta\right) {m}_{d - 1}}{\alpha^{2} + 2 \alpha \beta + \alpha d - \alpha + \beta^{2} + \beta d - \beta}.
\end{multline}
The last expression is linear in $c$, with negative coefficient (because $m_{d-1}\geqslant 0$ and $d\geqslant 3$). Thus, it is non-positive if and only if the condition
\begin{align}
c\geqslant  \frac{\beta - \alpha}{\left(\alpha + \beta\right) \left(\alpha + \beta + d - 1\right)}
\end{align}
holds for all $d\geqslant 3$. Since $\beta-\alpha\geqslant 0$, this is equivalent to
\begin{align}
c\geqslant  \frac{\beta - \alpha}{\left(\alpha + \beta\right) \left(\alpha + \beta + 2\right)},
\end{align}
and under this condition, \eqref{eq:operator_coeffs} is non-positive. Replacing in this reasoning $c$ with $\frac{c}{2}$, to align with \Cref{thm:beta_bernstein}, finishes the proof.

The calculations are done in \texttt{Sympy} and shown in \Cref{alg:simplify_diff_op}.

\begin{lstlisting}[caption={Functional inequality on the moment generating function.},label=alg:simplify_diff_op]
import sympy as sm
from IPython.display import display

v,c,a,b,d = sm.symbols('v,c,alpha,beta,d',nonnegative=True)
m = sm.IndexedBase('m',real=True)

# differential operator
L = d*m[d]-(d-1)*c*m[d-1]-v*m[d-2]

# simplify by the recursion - eliminate term d-2
L0 = d*(a+b+d-1)*m[d] - (d-1)*(b-a)/(a+b)*m[d-1] - a*b/(a+b)**2*m[d-2]
subs = {v:a*b/(a+b)**2*1/(a+b+1)}
subs.update( {m[d-2]:sm.solve(L0,m[d-2])[0]} )
L = L.subs(subs)

# print the simplified operator
display(L.expand().coeff(m[d]).factor())
display(L.expand().coeff(m[d-1]).factor(c))

# operator upper-bounds (skipping term with  'm[d-2]' when non-negative)
md_lbound = sm.solve(L0.subs({m[d-2]:0}),m[d])[0].factor()
display(md_lbound)
L_ubound = L.subs({m[d]:md_lbound}).factor().collect(c)
L_ubound = L_ubound.as_poly([c,m[d-1]]).as_expr()
display(L_ubound)
display(L_ubound.subs({c:0}).factor())

# solve for c which makes the upper-bound non-positive
R = L_ubound.coeff(m[d-1])
c_solution = sm.solve(R,c)[0].factor()
display(c_solution)
\end{lstlisting}
\end{proof}
We are now ready to estimate the cumulant generating function:
\begin{claim}\label{claim:cumulant_bound}
For non-negative $t$ and $c$ as in \Cref{claim:diff_cumulant_bound} we have
\begin{align}
\psi(t) \leqslant 
\begin{cases}
-v \cdot \frac{ct+\log(1-ct)}{c^2} & \alpha\leqslant \beta, \quad t<\frac{1}{c} \\
\frac{v t^2}{2} & \alpha > \beta.
\end{cases}
\end{align}
\end{claim}
\begin{proof}
This follows from the previous claim, by integrating the first branch of the bound and, respectively, by taking the logarithm of the second branch.
\end{proof}
\begin{remark}
Note that this bound, in case $\alpha\leqslant \beta$, gives $\psi(t)\leqslant 1+\frac{t^{2} v}{2} + \frac{(\beta-\alpha) t^{3} v}{3(\alpha+\beta+2)} + O\left(t^{4}\right)$, and comparing this with the actual expansion $\psi(t) = 1+\frac{t^2\mathbf{Var}[X]}{2}+\frac{t^3\mathbf{E}[(X-\mathbf{E}[X])^3]}{6}+O(t^4)$ we find that the bound is sharp up to $O(t^4)$.
\end{remark}

Finally, we present the Chernoff inequality:
\begin{claim}\label{claim:cumulant_optimisation}
For non-negative $\epsilon$ and $c$ as in \Cref{claim:diff_cumulant_bound} we have
\begin{align}
\psi^{*}(\epsilon) \geqslant 
\begin{cases}
\frac{\epsilon^{2}}{2 v \left(1+\frac{2c \epsilon}{3 v} \right)} & \alpha\leqslant \beta \\
\frac{\epsilon^2}{2v} & \alpha>\beta.
\end{cases}
\end{align}
\end{claim}
\begin{proof}
With the help of a computer algebra package, we plug the upper bound on $\psi(t)$ from \Cref{claim:cumulant_bound}
 into the Crammer-Chernoff bound, and find that it is maximized at $t_{best} = \frac{\epsilon}{c \epsilon + v}$  when $0<t<1/c$. This yields the bound
\begin{align}
\begin{aligned}
\psi^{*}(\epsilon) &\triangleq \sup\{t \epsilon- \psi(t) : 0\leqslant t \} \\
& \geqslant \epsilon t_{best} - \psi(t_{best}) \\
& = \frac{v}{c^2}\cdot \left( \frac{c \epsilon}{v} - \log{\left(\frac{c \epsilon}{v} + 1 \right)}\right).
\end{aligned}
\end{align}
The computations are done in \texttt{Sympy}, as shown in \Cref{alg:crammer_chernoff_solution}.
\begin{lstlisting}[caption={Optimizing the Crammer-Chernoff Bound.},label={alg:crammer_chernoff_solution}]
import sympy as sm
v,c,t,eps,x = sm.symbols('v,c,t,epsilon,x',positive=True)

log_MGF_bound = -v*(c*t+sm.log(1-c*t))/c**2
chernoff_crammer = eps*t-log_MGF_bound
# define t_best as the critical point
t_best = sm.solve(chernoff_crammer.diff(t,1),t)[0]
t_best
# show that t_best is a maximum 
assert not chernoff_crammer.diff(t,2).simplify().is_positive

# compute the Chernoff-Crammer bound, and simplify
exp_best = -chernoff_crammer.subs({t:t_best}).simplify()
exp_best = exp_best.subs({c*eps:x*v}).simplify().subs({x:c*eps/v},evaluate=False)
\end{lstlisting}
Using \eqref{eq:log_inequality} with $x=\frac{c \epsilon}{v}$ we finally obtain
\begin{align}
\psi^{*}(\epsilon) \geqslant \frac{\epsilon^{2}}{2 v \left(1+\frac{2c \epsilon}{3 v} \right)},
\end{align}
which finishes the proof.
\end{proof}
\Cref{thm:beta_bernstein} follows now by
\eqref{eq:chernoff_bound} and replacing  $c=\frac{\beta-\alpha}{(\alpha+\beta)(\alpha+\beta+2)}$ by $2c$. 


\subsection{Proof of \Cref{thm:lower_bound}}

Fix $\epsilon \geqslant 0$. Denote $Z=X-\mathbf{E}[X]$, and let $\phi(t) = \mathbf{E}[\exp(tZ)]$, $\psi(t) = \log \phi(t)$ be, respectively, the moment and cumulant generating function of $Z$. 
By Jensen's inequality $\phi(t) \geqslant \exp(t\mathbf{E}[Z])=1$, and so $\psi(t) \geqslant 0$ and $t\epsilon-\psi(t)\leqslant 0$ for $t\leqslant 0$. Since $\psi(0)=0$ we conclude that
\begin{align}\label{eq:optimal_inequality}
\psi^{*}(t) = \sup \left\{t\epsilon - \psi(t): t\geqslant 0\right\} = \sup \left\{t\epsilon - \psi(t): t \in \mathbb{R}\right\},
\end{align}
which shows that the Cram\'{e}r-Chernoff exponent equals the global maximum of the function $t\rightarrow t\epsilon - \psi(t)$ (Legendre-Fenchel
transform), achieved for $t\geqslant 0$.

It is known that the cumulant generating function is convex, and strictly convex for non-constant random variables~\cite{boucheron2003concentration}. Thus, $\psi(t)$ is strictly convex. As a consequence, we see that the function $t\rightarrow t\epsilon - \psi(t)$ is strictly concave. Thus, \eqref{eq:optimal_inequality} is maximized at the value $t$ which is a unique solution to
\begin{align}
\psi'(t)-\epsilon = 0.
\end{align}

The last equation defines $t=t(\epsilon)$ implicitly and seems to be not solvable with elementary functions. Nevertheless, the Lagrange Inversion Theorem 
ensures that $t$ is analytic as a function of $\epsilon$ and gives initial terms of the series expansion. 
More precisely, if $y = f(x)$ where $f$ is analytic at $0$ and $f(0)=0$, $f'(0)\not=0$ then $x$ expands into a power series in $y$ around $y=0$,
with the coefficient of the term $y^k$ equal to $\frac{1}{k}$ times the coefficient of $x^{k-1}$ in the series expansion of $\left(x/f(x)\right)^k$ when $k>0$ and $0$ when $k=0$;
for a reference, see for example~\cite{flajolet_sedgewick_2009}.
Applying this to $y=\epsilon$, $x=t$ and $f=\psi'$ we obtain
\begin{align}
t =\frac{\epsilon}{v}  - \frac{c \epsilon^{2}}{2 v^{2}}  + O(\epsilon^3),
\end{align}
where we do computations in \texttt{Sympy} as shown in \Cref{alg:best_chernoff}. Note that we use only up to three terms of the expansion of $f$, 
since we are interested in the first two terms for $x$.
\begin{lstlisting}[caption={Best possible Chernoff-type tail bounds.},label={alg:best_chernoff}]
import sympy as sm
from IPython.display import display

t,eps,v,c = sm.symbols('t,epsilon,v,c')

MGF = 1 + t**2*v/2 + t**3*v*c/6 
log_MGF = sm.log(MGF)
d_log_MGF = log_MGF.diff(t)
crammer_chernoff = eps*t-log_MGF

ks = [1,2]
t_eps_coeffs = [ 1/sm.factorial(k)*((t/d_log_MGF)**k).diff(t,k-1).series(t,0,1).subs({t:0}) for k in ks]
t_eps = sum( eps**k*c.simplify() for (c,k) in zip(t_eps_coeffs,ks) )
t_eps = t_eps.series(eps,0,3)
display(sm.Eq(sm.symbols('t_eps'),t_eps))

bound_best = crammer_chernoff\
  .subs({t:t_eps})\
  .series(eps,0,4)

display(sm.Eq(sm.symbols('\psi^{*}(\epsilon)'),bound_best))
\end{lstlisting}

Therefore, we obtain the following exponent in the Chernoff inequality:
\begin{align}
\psi^{*}(\epsilon) = t\epsilon - \psi(t) = \frac{\epsilon^{2}}{2 v} - \frac{c \epsilon^{3}}{6 v^{2}} + O\left(\epsilon^{4}\right),
\end{align}
which finishes the proof of the formula in \Cref{thm:lower_bound}. This expansion matches the expansion of the exponent in \Cref{thm:beta_bernstein} up 
to $O(\epsilon^4)$, which proves the optimality of the numeric constants.



\section{Conclusion}\label{sec:conclude}

This work established the closed-form sub-gamma type (Bernstein's) concentration bound for the beta distribution, with optimal constants. 
This solves the challenge settled by the prior work on sharp sub-gaussian bounds.


\section*{Acknowledgements}
The author is grateful to the anonymous reviewers for their constructive comments on the manuscript and help with correcting the proofs, and to Lutz Duembgen for insightful technical discussions.




\bibliographystyle{amsplain}
\bibliography{citations}
\end{document}